\documentclass[10pt]{amsart}

\usepackage{lmodern}
\usepackage[T1]{fontenc}

\usepackage{amsmath, amsthm, amscd, amssymb, xcolor}
\definecolor{dblue}{rgb}{0,0,.6}
\usepackage[colorlinks=true, linkcolor=dblue, citecolor=dblue, filecolor = dblue, menucolor = dblue, urlcolor = dblue]{hyperref}
\usepackage{tikz-cd} 
\usetikzlibrary{arrows}

\usepackage[all,cmtip]{xy}

\newcommand{\bbZ}{\mathbb{Z}}
\newcommand{\bbQ}{\mathbb{Q}}
\newcommand{\bbR}{\mathbb{R}}
\newcommand{\bbC}{\mathbb{C}}
\newcommand{\bbH}{\mathbb{H}}

\newcommand{\bbP}{\mathbb{P}}
\renewcommand{\AA}{\mathcal{A}}
\newcommand{\CC}{\mathcal{C}}
\newcommand{\DD}{\mathcal{D}}

\newcommand{\HH}{\mathcal{H}}
\newcommand{\LL}{\mathcal{L}}
\newcommand{\MM}{\mathcal{M}}

\newcommand{\OO}{\mathcal{O}}

\newcommand{\VV}{\mathcal{V}}

\newcommand{\XX}{\mathcal{X}}
\newcommand{\YY}{\mathcal{Y}}
\newcommand{\ZZ}{\mathcal{Z}}

\newcommand{\frgt}{\mathfrak{g}_{\mathrm{tot}}}

\newcommand{\frsl}{\mathfrak{sl}}

\newcommand{\frso}{\mathfrak{so}}

\renewcommand{\ge}{\geqslant}

\newcommand{\fldk}{\mathsf{k}}

\newcommand{\catMmot}{\mathsf{(Mot_{A})}}
\newcommand{\catMab}{\mathsf{(Mot_A^{ab})}}

\newcommand{\catVarC}{\mathsf{(Var/}\mathbb{C}\mathsf{)}}
\newcommand{\catAlg}[1]{\mathsf{(GrAlg/}#1\mathsf{)}}

\newcommand{\bfM}{\mathbf{M}}
\newcommand{\bfH}{\mathbf{H}}

\newcommand{\st}{\enskip |\enskip}
\newcommand{\sdot}{{\raisebox{0.16ex}{$\scriptscriptstyle\bullet$}}}

\newcommand{\emrp}{\mathrm{End}}

\newcommand{\im}{\mathrm{im}}

\newcommand{\ii}{i}

\newcommand{\Cl}{\mathcal{C}l}

\newcommand{\Spin}{\mathrm{Spin}}
\newcommand{\lrarr}{\longrightarrow}
\newcommand{\hrarr}{\hookrightarrow}
\newcommand{\Gal}{\mathrm{Gal}}

\newtheorem{defn}{Definition}[section]
\newtheorem{prop}[defn]{Proposition}
\newtheorem{thm}[defn]{Theorem}
\newtheorem{lem}[defn]{Lemma}
\newtheorem{cor}[defn]{Corollary}

{\theoremstyle{remark}

}

\makeatletter
\@addtoreset{equation}{section}
\makeatother

\title{Deformation principle and Andr\'e motives of projective hyperk\"ahler manifolds}

\author{Andrey Soldatenkov}
\address{Institut f\"ur Mathematik, Humboldt-Universit\"at zu Berlin, Unter den Linden 6, 10099 Berlin}
\email{soldatea@hu-berlin.de}
\date{\today}

\thanks{}

\begin{document}

\begin{abstract}
Let $X_1$ and $X_2$ be deformation equivalent projective hyperk\"ahler manifolds.
We prove that the Andr\'e motive of $X_1$ is abelian if and only if the Andr\'e motive of $X_2$ is abelian.
Applying this to manifolds of $\mbox{K3}^{[n]}$, generalized Kummer and OG6 deformation types,
we deduce that their Andr\'e motives are abelian. As a consequence, we prove that all Hodge
classes in arbitrary degree on such manifolds are absolute. We discuss applications to the Mumford-Tate conjecture,
showing in particular that it holds for even degree cohomology of such manifolds.
\end{abstract}

\maketitle

\tableofcontents

\section{Introduction}

In this paper, we study Andr\'e motives of projective hyperk\"ahler manifolds.
By a hyperk\"ahler manifold we mean a compact simply connected K\"ahler manifold $X$
such that $H^0(X,\Omega^2_X)$ is spanned by a symplectic form.
We generalize the results of \cite{A2} and \cite{Sc}, showing that for most of
the known deformation types of hyperk\"ahler manifolds their Andr\'e motives are abelian.

\subsection{Andr\'e motives and hyperk\"ahler manifolds}
Andr\'e motives were introduced in \cite{A1} as a refinement of Deligne's motives \cite{DMOS}.
They form a semi-simple Tannakian category, whose construction we briefly recall
in section \ref{sec_mot}. The motives of abelian varieties generate a full Tannakian subcategory
whose objects are called abelian motives. The theory of Andr\'e motives has found
applications to the study of various arithmetic and Hodge-theoretic questions
about algebraic varieties. Recall the theorem of Deligne \cite{DMOS}
stating that any Hodge cohomology class on an abelian variety is absolute Hodge.
More generally, it was shown in \cite{A1} that for a projective manifold $X$
whose Andr\'e motive is abelian, all Hodge classes on $X$ are absolute.
This shows that one part of the Hodge conjecture holds for varieties
with abelian motives. Another application is related to the Mumford-Tate conjecture,
which predicts a relation between the Mumford-Tate groups and the Galois group
action on the cohomology of a projective variety. We review the absolute
Hodge classes and the Mumford-Tate conjecture in more detail in sections \ref{sec_abs}
and \ref{sec_mt} below. We recommend \cite{M1} for a general overview of
the recent developments in this area. We remark that the Mumford-Tate conjecture
for all known deformation classes of hyperk\"ahler manifolds was independently
proven in \cite{FFZ}, the proof relying on Theorem \ref{thm_main} below.

The main new tool used in this paper is the generalized Kuga-Satake
construction for hyperk\"ahler manifolds, which was introduced in \cite{KSV}. For any projective hyperk\"ahler
manifold $X$, this construction gives an embedding of the cohomology groups
of $X$ into the cohomology of an abelian variety, which respects the Hodge structures.
Therefore, our main goal is to prove that the Kuga-Satake embedding lifts
to the category of Andr\'e motives. To do this, we need to show
that the cohomology class defining the embedding is motivated in the sense
of \cite{A1}, see also section \ref{sec_mot}.

%The same method, using the classical Kuga-Satake construction, was applied
%in \cite{A2} to K3 surfaces. However, unlike \cite{A2}, 
Our approach is
based on the deformation principle for motivated cohomology classes \cite[Theorem 0.5]{A1}.
More precisely, assume that
\begin{equation}\label{eqn_pi}
\pi\colon \XX\to B
\end{equation}
is a smooth projective morphism, $B$ a connected quasi-projective variety
and the fibres of $\pi$ are hyperk\"ahler manifolds. Assume that for some
point $b_0\in B$ the Andr\'e motive of $\XX_{b_0} = \pi^{-1}(b_0)$ is abelian.
This implies that the Kuga-Satake embedding for $\XX_{b_0}$ is motivated,
and the deformation principle \cite[Theorem 0.5]{A1} implies that
it is motivated for any fibre $\XX_{b_1}$, $b_1 \in B$. Therefore, it 
suffices to prove that there is one hyperk\"ahler manifold in each
deformation class whose motive is abelian. This is usually possible
to do using some explicit geometric construction. For example, in
the case of K3 surfaces one can assume that $\XX_{b_0}$ is a Kummer surface.
Other deformation types of hyperk\"ahler manifolds are discussed in section \ref{sec_main}.

The approach outlined above has one subtlety. Namely, assume that $X_1$
and $X_2$ are deformation equivalent projective hyperk\"ahler manifolds.
In the moduli space of all hyperk\"ahler manifolds the projective ones
are parametrized by a countable collection of divisors. We show in
section \ref{sec_mot_hk} that it is possible to realize $X_1$ and $X_2$
as fibres of a smooth analytic family as in (\ref{eqn_pi}), but it is a priori
not clear if one can make the family algebraic.
%The basic difficulty
%is that the specialization of an ample line bundle need not remain ample,
%or even nef.
To resolve this issue, we prove in section \ref{sec_def}
a generalization of the deformation principle, which applies to the case
when all fibres of $\pi$ in (\ref{eqn_pi}) are projective, but the morphism
$\pi$ is projective only over a dense Zariski-open
subset of the base.

Before stating our main result we recall that two compact hyperk\"ahler
manifolds $X_1$ and $X_2$ are called deformation equivalent if they can
be realized as two fibres of a smooth family (\ref{eqn_pi}) where $B$ is
a connected complex analytic space and all fibres of $\pi$ are compact
hyperk\"ahler manifolds. So $X_1$ and $X_2$ are deformations of each other
in the complex analytic sense, and in general no polarization is preserved
along the deformation. Our main result is the following statement.

\begin{thm}\label{thm_main}
Let $X_1$ and $X_2$ be deformation equivalent projective hyperk\"ahler manifolds.
The Andr\'e motive of $X_1$ is abelian if and only if the Andr\'e motive of
$X_2$ is abelian.
\end{thm}

The proof is given in section \ref{sec_thm}. 

We apply this theorem to several known deformation types of hyperk\"ahler manifolds.
We leave out only the OG10 type, which has recently been treated by Floccari, Fu and Zhang in \cite{FFZ}.
%who completed the investigation of Andr\'e motives of projective hyperk\"ahler manifolds
%of the known deformation types and provided alternative proves of some of our results.

\begin{cor}\label{cor_hk}
Let $X$ be a projective hyperk\"ahler manifold of ${\rm K3}^{[n]}$, generalized Kummer, or OG6
deformation type. Then the Andr\'e motive of $X$ is abelian.
\end{cor}
\begin{proof}
By Theorem \ref{thm_main}, it suffices to find one manifold with
abelian motive in each deformation class.
For the Hilbert schemes of points on K3 surfaces and generalized Kummer varieties,
the motives are abelian by \cite{DM1}, \cite{DM2} and \cite{Xu}.
For OG6 deformation type, one can find a manifold with abelian motive
using the construction from \cite{MRS}.
We recall all these constructions in section \ref{sec_main}.
\end{proof}

The above Corollary recovers the results of \cite{A1}, \cite{A2} and \cite{Sc},
where the case of K3 surfaces and, more generally, ${\rm K3}^{[n]}$-type
varieties has been considered. 
Unlike \cite{Sc}, we do not
use the results of Markman on the structure of the cohomology ring,
which are specific for ${\rm K3}^{[n]}$-type varieties. The approach using
the Kuga-Satake construction is more general, and therefore allows us to treat
other deformation types. 

Let us also mention that a substantial amount of recent research has been
devoted to the study of Chow motives of hyperk\"ahler manifolds and related
questions, see e.g. \cite{FV} and references therein. Proving that
Chow motives of hyperk\"ahler manifolds are abelian seems to be a more difficult
problem, and the methods of the present paper are not sufficiently strong
to deal with it.

Let us next review
the applications of our results to the absolute Hodge classes and the Mumford-Tate
conjecture.

\subsection{Absolute Hodge classes}\label{sec_abs}
Let $X$ be a non-singular projective variety over $\bbC$.
Denote by $X^{an}$ the corresponding complex manifold. Recall that the de Rham cohomology
$H^\sdot_{dR}(X)$ is the hypercohomology of the algebraic de Rham complex $\Omega_{X/\bbC}^\sdot$.
For every $k$, the $\bbC$-vector space $H^k_{dR}(X)$ is endowed with a decreasing Hodge filtration $F^\sdot$.
On the other hand, the singular cohomology $H^k(X^{an},\bbC)$ is endowed with the $\bbQ$-structure
given by the subspace $H^k(X^{an},\bbQ)$. Comparison results between the algebraic and the analytic
cohomology of coherent sheaves and the quasi-isomorphism $\Omega^\sdot_{X^{an}}\simeq \bbC$ induce
natural isomorphisms of the cohomology groups $H^{k}_{dR}(X)\simeq H^k(X^{an},\bbC)$.

Recall that an element $\alpha\in F^pH^{2p}_{dR}(X)$ is called a Hodge class, if its image
in $H^{2p}(X^{an},\bbC)$ under the isomorphism described above is contained in the subspace
$(2\pi i)^pH^{2p}(X^{an},\bbQ)$. The Hodge conjecture implies that the property of $\alpha$
being Hodge should be stable under automorphisms of the field of complex numbers.
More precisely, let $\sigma\in \mathrm{Aut}(\bbC/\bbQ)$ be an automorphism and $X_\sigma$ be the
variety obtained by base change via $\sigma$. We have a chain of isomorphisms of $\bbQ$-vector spaces:
\begin{equation}\label{eqn_cohom}
H^{k}_{dR}(X) \simeq H^{k}_{dR}(X)\otimes_{\bbC,\sigma}\bbC\simeq H^{k}_{dR}(X_\sigma) \simeq H^{k}(X_{\sigma}^{an},\bbC).
\end{equation}

A cohomology class $\alpha\in F^pH_{dR}^{2p}(X)$ is called absolute Hodge, if its image
under the isomorphism (\ref{eqn_cohom}) lies in $(2\pi i)^p H^{2p}(X_{\sigma}^{an},\bbQ)$ for any
$\sigma\in \mathrm{Aut}(\bbC/\bbQ)$. If $\alpha$ is an algebraic class, i.e. it is
contained it the $\bbQ$-subspace spanned by the classed of algebraic subvarieties, then it is absolute
Hodge. According to the Hodge conjecture, every Hodge class should be algebraic, therefore absolute Hodge.

According to Deligne \cite{DMOS}, any Hodge class on an abelian variety is absolute.
Using the results of \cite{A1} and Theorem \ref{thm_main}, we deduce the following statement.

\begin{cor}\label{cor_abs}
Let $X$ be a projective hyperk\"ahler manifold of ${\rm K3}^{[n]}$, generalized Kummer, or OG6
deformation type. Then all Hodge classes on $X$ are absolute.
\end{cor}
\begin{proof}
By Corollary \ref{cor_hk}, the Andr\'e motive of $X$ is abelian, and
we can apply \cite[section 6]{A1}.
\end{proof}

\subsection{The Mumford-Tate conjecture}\label{sec_mt}
The purpose of this section is to explain how the results of Floccari \cite{Fl}
combined with Theorem \ref{thm_main} lead to the proof of some
cases of the Mumford-Tate conjecture for hyperk\"ahler manifolds.
A similar proof was independently found in \cite{FFZ}.

Assume that $X$ is a non-singular projective variety defined
over a subfield $\fldk\subset \bbC$ finitely generated
over $\bbQ$. Recall the comparison isomorphism between the $\ell$-adic and
singular cohomologies of $X$:
\begin{equation}\label{eqn_ladic}
H^k_{\acute{e}t}(X_{\bar{\fldk}},\bbQ_\ell)\simeq H^k(X^{an},\bbQ)\otimes \bbQ_\ell.
\end{equation}
Let us briefly denote by $H^k_\ell$ this $\bbQ_\ell$-vector space.

The left-hand side of (\ref{eqn_ladic}) is a representation
of the Galois group $\Gal(\bar{\fldk}/\fldk)$.
Denote by $G_\ell^k$ the Zariski closure of the image
of $\Gal(\bar{\fldk}/\fldk)$ in $\mathrm{GL}(H^k_\ell)$.
Let $G_\ell^{k,\circ}$ be the connected component of the identity in $G_\ell^k$.
The right-hand side of (\ref{eqn_ladic}) is naturally a representation
of the Mumford-Tate group $\mathrm{MT}^k(X)\otimes \bbQ_\ell$.
The Mumford-Tate conjecture predicts
that these two subgroups of $\mathrm{GL}(H^k_\ell)$ are equal:
\begin{equation}\label{eqn_mt}
G_\ell^{k,\circ} = \mathrm{MT}^k(X)\otimes \bbQ_\ell.
\end{equation}

This conjecture has been a subject of an active research. For an
overview of the recent developments, see \cite{M1}. There is a number
of recent works on the Mumford-Tate conjecture that use methods
similar to ours. In \cite{A2}, the Mumford-Tate conjecture in
degree 2 was proven for hyperk\"ahler manifolds. In \cite{M2} 
this result was generalized to a wider class of varieties with $h^{2,0} = 1$;
the proof relies on the Kuga-Satake construction. In \cite{C1}, it was
shown that for varieties with abelian Andr\'e motive the validity
of the Mumford-Tate conjecture does not depend on $\ell$. In \cite{Fl},
the Mumford-Tate conjecture in arbitrary degree for ${\rm K3}^{[n]}$-type
varieties was proven; the method of the proof relies on the results of
Markman about the structure of the cohomology algebra, similarly to \cite{Sc}.

Let us remark, that at present the Mumford-Tate conjecture is not
known even for general abelian varieties. On the other hand, we know from
\cite{A2} that it holds in degree 2 for any hyperk\"ahler manifold. Using
Theorem \ref{thm_main} and the results of \cite{Fl}, we can deduce
the Mumford-Tate conjecture in all even degrees for the same type
of varieties as in Corollary \ref{cor_hk}.

One special type of abelian varieties for which the Mumford-Tate conjecture
is known are the varieties of CM type. In this case, one can use the
results of \cite{Va}, see also \cite[Theorem 3.3.2 and Corollary 4.3.15]{M1}.
We deduce analogous results for hyperk\"ahler manifolds.
We will say that a projective hyperk\"ahler manifold $X$ is of CM type,
if the Mumford-Tate group of $H^2(X,\bbQ)$ is abelian. In this case the Mumford-Tate
groups of $H^k(X,\bbQ)$ are abelian for all $k$. This follows from the fact that the
Hodge structures on all cohomology groups of $X$ are induced by the natural Lie algebra
action, as we recall in section \ref{sec_ks}.

We summarize our discussion in the following statement. The idea of its proof is entirely
due to Floccari, and we follow his arguments from \cite{Fl}. The only
ingredient that was missing in \cite{Fl} is the deformation principle, Theorem \ref{thm_main}.
In the case of ${\rm K3}^{[n]}$-type it was obtained in \cite{Fl} independently.

\begin{cor}\label{cor_MT}
Let $X$ be a projective hyperk\"ahler manifold of ${\rm K3}^{[n]}$, generalized Kummer, or OG6
deformation type.
Then the Mumford-Tate conjecture holds for the cohomology of $X$ in all
even degrees. If $X$ is moreover of CM type, then the Mumford-Tate
conjecture holds for the cohomology in all degrees.
\end{cor}

\begin{proof}
Let $H^+_\ell = \oplus_{k}H^{2k}_\ell(X^{an},\bbQ_\ell)$ and $H^2_\ell = H^{2}(X^{an},\bbQ_\ell)$.
Denote by $G^{+,\circ}_\ell \subset \mathrm{GL}(H^+_\ell)$ and $G^{2,\circ}_\ell \subset \mathrm{GL}(H^2_\ell)$
the connected algebraic groups obtained from the Galois representations as
explained above. Let $\mathrm{MT}^+ \subset \mathrm{GL}(H^+_\ell)$ and $\mathrm{MT}^2 \subset \mathrm{GL}(H^2_\ell)$
be the Mumford-Tate groups. Since the Andr\'e motive of $X$ is abelian, by
\cite[formula (3.3)]{M1} we have an inclusion $G^{+,\circ}_\ell\subset \mathrm{MT}^+$.
By the work of Andr\'e \cite{A2}, the analogous inclusion in degree two
is an isomorphism $G^{2,\circ}_\ell\simeq \mathrm{MT}^2$. We have the following
diagram of groups
\begin{equation}
\begin{tikzcd}[]
G^{+,\circ}_\ell \arrow[hook]{r}\arrow[two heads]{d} & \mathrm{MT}^+ \arrow{d}{\simeq} \\
G^{2,\circ}_\ell \arrow[hook]{r}{\simeq} & \mathrm{MT}^2
\end{tikzcd}\nonumber
\end{equation}
where the vertical arrows are induced by the projection $H^+_\ell \to H^2_\ell$.
The arrow on the right is an isomorphism because the Hodge structures on
the cohomology of $X$ are induced by the action of the orthogonal Lie algebra $\frgt$
(see section \ref{sec_ks} and \cite[(1.7)]{LL}).
% The representation of Deligne's torus
%that determines the Hodge structures on $H^+$ factors through the induced action of
%the orthogonal group. The Mumford-Tate group is therefore a subgroup of this orthogonal group,
%in particular it is the same in all even degrees.

It follows from the diagram above that the upper arrow is also surjective, which proves the first
part of the corollary. The second part follows
from \cite[Theorem 3.3.2]{M1}, see also \cite[Corollary 4.3.15]{M1}.
\end{proof}

\subsection{Organization of the paper} In section \ref{sec_mot}, we discuss
the notions of a motivated cohomology class and Andr\'e motive. We recall
all necessary definitions and constructions from \cite{A1}. In section \ref{sec_ks},
we recall basic results about the cohomology of hyperk\"ahler manifolds
and the Kuga-Satake construction from \cite{KSV}. The main results of this
section are Proposition \ref{prop_mc} and Corollary \ref{cor_cover}, which
generalize the results of \cite{KSV} to the relative setting.
In section \ref{sec_main}, we recall some known results about
the Hilbert schemes of points, generalized Kummer varieties and O'Grady's
6-dimensional varieties. We explain that in each of these deformation types
one can find a variety with abelian motive.
In section \ref{sec_def}, we prove a generalization of the
deformation principle for motivated cohomology classes, Proposition \ref{prop_def}.
In section \ref{sec_mot_hk}, we discuss the construction of families
of hyperk\"ahler manifolds and in \ref{sec_thm} prove the main result, Theorem \ref{thm_main}.

\section{Motivated cohomology classes and Andr\'e motives}\label{sec_mot}

In this section, we briefly recall some results of \cite{A1}. Let $\catVarC$ be the category of
non-singular complex projective varieties and their morphisms. For a variety $X\in\catVarC$,
we will denote by $H^k(X)$ the singular cohomology group $H^k(X^{an},\bbC)$.
Let $\catAlg{\bbQ}$ be the the category of finite-dimensional graded $\bbQ$-algebras.

\subsection{Motivated cohomology classes}
Assume that $X\in\catVarC$ and let $\LL$ be an ample line bundle on $X$.
Denote by $h\in H^{2}(X)$ the first Chern class of $\LL$.
The Lefschetz operator $L_h\in \emrp(H^\sdot(X))$ is defined as the cup product with $h$,
and it induces isomorphisms $L_h^k\colon H^{n-k}(X)\stackrel{\sim}{\to} H^{n+k}(X)$
for every $k=0,\ldots,n$, where $n=\dim_\bbC(X)$. The subspace of primitive elements $H^{k}_{pr}(X)\subset H^{k}(X)$,
$k = 0,\ldots,n$, 
is by definition the kernel of $L_h^{n-k+1}$.

We will denote by $*_h\in \emrp(H^\sdot(X))$ the Lefschetz involution associated with $h$.
Recall that for $x\in H^{k}_{pr}(X)$ and $i=0,\ldots,n-k$ we have $*_h(L_h^ix) = L_h^{n-k-i}x$.
This uniquely determines $*_h$, since $H^{\sdot}(X)$ is spanned by the elements of the form $L_h^ix$
with $x$ primitive.

For $X, Y\in \catVarC$ and two ample line bundles $\LL_1\in \mathrm{Pic}(X)$,
$\LL_2\in\mathrm{Pic}(Y)$, let $p_X$, $p_Y$ denote the two projections from $X\times Y$,
and let $h = c_1(p_X^*\LL_1\otimes p_Y^*\LL_2)\in H^2(X\times Y)$. For any two
classes of algebraic cycles $\alpha,\beta\in H^\sdot_{alg}(X\times Y)$, consider the class
\begin{equation}\label{eqn_mot}
p_{X*}(\alpha\cup *_h\beta).
\end{equation}

Let $H^\sdot_{M}(X)$ be the $\bbQ$-subspace of $H^\sdot(X)$ spanned by the classes (\ref{eqn_mot})
for all $Y$, $\LL_1$, $\LL_2$, $\alpha$, $\beta$ as above. Elements of $H^\sdot_M(X)$ will
be called motivated cohomology classes (or motivated cycles, in the terminology of \cite{A1}).
Let us list a few properties of the motivated classes.
\begin{enumerate}
\item $H^\sdot_{M}(X)$ is a graded $\bbQ$-subalgebra of $H^\sdot(X)$;
\item For $f\colon X\to Y$, we have $f^*H^\sdot_{M}(Y)\subset H^\sdot_{M}(X)$,
hence we have a  functor $$H^\sdot_{M}\colon \catVarC^{op}\to\catAlg{\bbQ};$$
\item For $f$ as above, we have $f_*H^\sdot_{M}(X)\subset H^{\sdot-\dim(X)}_{M}(Y)$;
\item All classes in $H^\sdot_{M}(X)$ are absolute Hodge;
\item The K\"unneth components of the diagonal are contained in $H^{\sdot}_M(X\times X)$ for any $X\in \catVarC$.
\end{enumerate}

Let us also recall the following deformation principle:
\begin{thm}[{\cite[Theorem 0.5]{A1}}]\label{def_princ}
Let $\pi\colon\XX\to B$ be a smooth projective morphism.
Assume that the base $B$ is a connected quasi-projective variety.
Let $\nu \in \Gamma(B,R^{2p}\pi_*\bbC)$.
Assume that for some $b_0\in B$ the element $\nu_{b_0}\in H^{2p}(\XX_{b_0})$ is a motivated
cohomology class. Then $\nu_{b_1}$ is a motivated class on $\XX_{b_1}$ for any $b_1\in B$.
\end{thm}
We will generalize the deformation principle in section \ref{sec_def}, relaxing the condition
of projectivity for the morphism $\pi$, see Proposition \ref{prop_def}.

\subsection{Andr\'e motives} The construction of Andr\'e motives is similar
to that of Chow motives. We start by defining the spaces of motivated correspondences:
for $X,Y\in \catVarC$ with $X$ connected, let $\mathrm{Cor}_M^k(X,Y) = H^{k+\dim(X)}_M(X\times Y)$.
The properties of motivated classes listed above imply that we can define the composition
of motivated correspondences in the usual way.

Define a $\bbQ$-linear category $\catMmot$ whose objects are triples $(X,p,n)$, where $X\in \catVarC$,
$p\in \mathrm{Cor}_M^0(X,X)$, $p\circ p = p$, and $n\in \bbZ$. The space of morphisms from $(X,p,n)$
to $(Y,q,m)$ is by definition $q\circ \mathrm{Cor}^{m-n}_M(X,Y)\circ p \subset \mathrm{Cor}^{m-n}_M(X,Y)$.
The tensor product on $\catMmot$ is defined using the Cartesian product of varieties.
It is shown in \cite{A1}, that $\catMmot$ is a semi-simple graded neutral Tannakian category.
We will denote the Andr\'e motives by boldface letters: for $X\in \catVarC$ define 
$$\bfM(X)(n) = (X,[\Delta_X],n)\in \catMmot,$$ where $\Delta_X$ denotes the diagonal in $X\times X$.
More generally, since the K\"unneth components of the diagonal are motivated, we can
define the motives representing the cohomology groups of $X$:
$$
\bfH^k(X)(n) = (X, \delta_k, n) \in \catMmot,
$$
where $\delta_k$ is the $k$-th K\"unneth component of $[\Delta_X]$. Therefore, we have
the direct sum decomposition $\bfM(X)(n) = \oplus_k \bfH^k(X)(n)$.

The subcategory of abelian motives $\catMab$ is the minimal full Tannakian subcategory of $\catMmot$
containing $\bfM(A)(n)$ for all abelian varieties $A$ and $n\in \bbZ$. It is shown in \cite{A1}
that for any $X\in\catVarC$, if $\bfM(X)\in \catMab$, then all Hodge classes on $X$ are motivated,
in particular absolute Hodge.

We will use the well-known fact that the class of varieties with abelian motives is stable under
two basic operations. Namely, let $X,Y\in \catVarC$ and let $f\colon X\to Y$ be a surjective
generically finite morphism. Then $f^*\colon \bfM(Y) \hrarr \bfM(X)$ is injective, and
$\bfM(X)\in \catMab$ implies $\bfM(Y)\in \catMab$. The second basic operation is the blow-up:
if $i\colon X\hrarr Y$ is a closed immersion and $\bfM(X), \bfM(Y)\in \catMab$, then
$\bfM(\mathrm{Bl}_XY)\in\catMab$. As an example, consider a complex abelian surface $A$
and denote by $A[2]$ its $2$-torsion points. Then $S = (\mathrm{Bl}_{A[2]}A)/\pm 1$ is
the corresponding Kummer K3 surface. We see that $\bfM(S) \in \catMab$.

% To prove that the motive of $X$ is abelian, it is enough to
%construct an abelian variety $A$ and for every $k$ to produce an embeddings of Hodge structures
%$H^k(X,\bbQ(k))\hrarr H^{m}(A,\bbQ(m))$ given by a motivated class on the product $X\times A$,
%for some $m$.

\section{Hyperk\"ahler manifolds and the Kuga-Satake embedding}\label{sec_ks}

In this section, we recall basic properties of hyperk\"ahler manifolds.
We refer to \cite{So3}, \cite{Hu1}, \cite{Be} for more details.
We also recall the results of \cite{KSV} and extend them to the relative setting.

\subsection{Hyperk\"ahler manifolds}\label{sec_hyperk}
In this paper, a hyperk\"ahler manifold is a compact simply connected
K\"ahler manifold $X$ such that $H^0(X,\Omega^2_X)$
is spanned by a symplectic form. The dimension of any symplectic
manifold is even; we let $\dim_\bbC(X) = 2n$.

Let $V_\bbZ=H^2(X,\bbZ)$ and $V = V_\bbZ\otimes\bbQ$. Let $q\in S^2V^*$ denote
the Beauville-Bogomolov-Fujiki (BBF) form. Recall that this form has the following
property: there exists a constant $c_X\in \bbQ$ such that for all
$h\in H^2(X,\bbQ)$ the equality $q(h)^n = c_X h^{2n}$ holds.
We can assume that $q$ is integral and primitive on $V_\bbZ$, and $q(h) > 0$ for a K\"ahler
class $h$. The signature of $q$ is $(3,b_2(X)-3)$.
Recall also that $V$ carries a Hodge structure of K3 type.

There exists a natural action of the orthogonal Lie algebra on the
total cohomology of $X$. Let us briefly recall how to define this action.
An element $h\in V$ has Lefschetz property, if the cup product with $h^k$
induces an isomorphism $H^{2n-k}(X,\bbQ)\simeq H^{2n+k}(X,\bbQ)$ for all $k=0,\ldots,2n$.
In this case one can consider the corresponding Lefschetz $\frsl_2$-triple.
Let $\frgt\subset \emrp(H^{\sdot}(X,\bbQ))$ be the Lie subalgebra generated
by all such $\frsl_2$-triples.

Let $\tilde{V} = \langle e_0\rangle\oplus V\oplus \langle e_4\rangle$
be the graded $\bbQ$-vector space with $e_i$ of degree $i$, and $V$ in degree 2.
Let $\tilde{q}\in S^2\tilde{V}^*$ be the quadratic form such that:
$\tilde{q}|_V = q$, $e_0$ and $e_4$ are isotropic and orthogonal to $V$ and span a hyperbolic
plane. It was shown in \cite{Ve1} and \cite{LL} that there exists an isomorphism
of graded Lie algebras $\frgt\simeq\frso(\tilde{V},\tilde{q})$.

One can show that Hodge structures on the cohomology groups of $X$
are induced by the action of $\frgt$. More precisely, let $W$ be the Weil operator
that induces the Hodge structure on $V$, i.e. it acts on $V^{p,q}$ as multiplication by $\ii(p-q)$.
Then $ W\in \frso(V,q)\subset \frso(\tilde{V},\tilde{q})$ and $W$ induces
the Hodge structures on $H^k(X,\bbQ)$ for all $k$ (see \cite{Ve1}).

\subsection{The Kuga-Satake construction}\label{subsec_ks}
To the K3 type Hodge structure $V$ we can associate a Hodge structure of abelian type,
which is called the Kuga-Satake Hodge structure.
Let us briefly recall the construction. Let $H = \Cl(V,q)$ be the Clifford algebra.
There exists a natural embedding $V\hrarr H$.
Define $H^{0,-1}$ to be the right ideal $V^{2,0}\cdot H_\bbC$ (see \cite[Lemma 3.3]{SS}), and $H^{-1,0} =\overline{H^{0,-1}}$.
This defines a rational Hodge structure on $H$. 

Note that $H$ is canonically an
$\frso(V,q)$-module, and the Hodge structure on it is induced by the action of the
Weil operator $W$ (see e.g. \cite{SS}). Since the hyperk\"ahler manifold $X$
is projective, one can show that $H$ is polarizable.
Moreover, the polarization can be chosen $\Spin(V,q)$-invariant (see e.g. \cite{KSV} or \cite{VG}).

Let $d = \frac{1}{4} m \dim_\bbQ(H)$. The following theorem was proven in \cite{KSV}.

\begin{thm}\label{thm_ks}
There exists a structure of graded
$\frso(\tilde{V},\tilde{q})$-module on $\Lambda^\sdot H^*$
that extends the $\frso(V,q)$-module structure.
For some $m>0$ there exists an
embedding of $\frso(\tilde{V},\tilde{q})$-modules
$$
H^{\sdot+2n}(X,\bbQ)\hrarr \Lambda^{\sdot+2d}(H^{* \oplus m}).
$$
This induces embeddings of Hodge structures
$$
\nu_i\colon H^{i+2n}(X,\bbQ(n))\hrarr \Lambda^{i+2d}(H^{* \oplus m})(d),
$$
where $i = -2n,\ldots,2n$.
\end{thm}

\subsection{The Kuga-Satake construction in families}

Let us consider a smooth projective morphism $\varphi\colon \XX\to B$
whose fibres are hyperk\"ahler manifolds. Let us fix a base point $b_0\in B$
and denote by $X$ the fibre $\XX_{b_0}$. We would like to apply
Theorem \ref{thm_ks} to the family $\varphi$ and obtain an embedding of
the corresponding variations of Hodge structures. In order to do so,
we need to construct a family of Kuga-Satake abelian varieties such
that the embeddings $\nu_i$ from Theorem \ref{thm_ks} are $\pi_1(B,b_0)$-equivariant.
We will explain below, that it is possible to do this after we pass
to a finite \'etale covering of $B$. The base $B$ can be an arbitrary
complex analytic space. 

Denote by $\mathrm{Aut}^P(X)\subset \mathrm{GL}(H^\sdot(X,\bbQ))$ the subgroup
of algebra automorphisms that fix the Pontryagin classes of $X$.
We will denote by $\mathrm{Aut}^P(X)_\bbZ$ the arithmetic subgroup
of $\mathrm{Aut}^P(X)_\bbQ$ that consists of all elements preserving the integral cohomology lattice.

Recall that $\frso(V,q)$ acts on $H^\sdot(X,\bbQ)$ by derivations,
and this action induces a homomorphism of algebraic groups
$$
\alpha\colon\Spin(V,q)\to \mathrm{Aut}^P(X),
$$
see section 3.1 in \cite{So1} and references therein.
Denote by $\mathrm{Aut}^+(X)$ the image of $\alpha$.
Let $\mathrm{MC}(X) = \mathrm{Diff}(X)/\mathrm{Diff}^\circ(X)$
be the mapping class group of $X$.
Here $\mathrm{Diff}(X)$ is the group of diffeomorphisms of $X$, and $\mathrm{Diff}^\circ(X)$ is
the subgroup of diffeomorphisms isotopic to the identity.
The mapping class group acts on the cohomology of $X$ fixing
Pontryagin classes, hence a group homomorphism 
$$
\beta\colon \mathrm{MC}(X)\to \mathrm{Aut}^P(X)_\bbZ.
$$

The following proposition generalizes \cite[Proposition 3.5]{So1}. Let
us mention that an analogous result was independently obtained in \cite[Theorem 4.16]{GKLR}.

\begin{prop}\label{prop_mc}
There exists a subgroup of finite index $\mathrm{MC}'\subset \mathrm{MC}(X)$ and
a group homomorphism $\beta'\colon \mathrm{MC}'\to \Spin(V,q)_\bbQ$ such that
$\alpha\circ \beta' = \beta$ and the image of $\beta'$ is contained in an arithmetic
subgroup of $\Spin(V,q)_\bbQ$.
\end{prop}
\begin{proof}
Recall that the action of $\mathrm{Aut}^P(X)$ on $V$ preserves the form $q$, thus we
have a homomorphism $\mathrm{Aut}^P(X)\to\mathrm{O}(V,q)$, see \cite[Theorem 3.5(i)]{Ve2}.
Let $\Gamma\subset \mathrm{Aut}^P(X)_\bbZ$ be a torsion-free arithmetic subgroup.
Since the induced homomorphism $\Gamma\to \mathrm{O}(V,q)_\bbQ$ has finite kernel by \cite[Theorem 3.5(iv)]{Ve2},
it is actually injective. Analogously, let $\Gamma'\subset \Spin(V,q)_\bbQ$ be another
torsion-free arithmetic subgroup. Then the action of $\Spin(V,q)$ on $V$ induces
a homomorphism $\Gamma'\hrarr \mathrm{O}(V,q)_\bbQ$. This homomorphism is injective because $\Gamma'$
is torsion-free and the kernel of the map $\Spin(V,q)_\bbQ\to \mathrm{O}(V,q)_\bbQ$ is finite.

The construction of $\mathrm{MC}'$ is summarized in the following commutative diagram:
\begin{equation}\label{eqn_groups}
\begin{tikzcd}[]
\mathrm{MC}' \arrow[hook]{r}\arrow{d} & \mathrm{MC''} \arrow{d}\arrow[hook]{r} & \mathrm{MC}(X)\dar{\beta} \\
\Gamma\cap\Gamma' \arrow[hook]{r}\arrow[hook]{d} & \Gamma \arrow[hook]{r} \arrow[hook]{d} & \mathrm{Aut}^P(X)_\bbZ \arrow{dl} \\
\Gamma' \arrow[hook]{r} & \mathrm{O}(V,q)_\bbQ
\end{tikzcd}\nonumber
\end{equation}
Here $\Gamma\cap\Gamma'$ is of finite index in $\Gamma$, because both $\Gamma$ and $\Gamma'$ are
arithmetic subgroups of $\mathrm{O}(V,q)_\bbQ$. By the definition of $\Gamma$, it has finite index in $\mathrm{Aut}^P(X)_\bbZ$.
We then let $\mathrm{MC}' = \beta^{-1}(\Gamma\cap\Gamma')$, and define $\beta'$ to be the
composition of the two maps in the left column of the diagram and the embedding $\Gamma'\subset \Spin(V,q)_\bbQ$.
\end{proof}

\begin{cor}\label{cor_cover}
Let $\varphi\colon \XX\to B$ be a smooth family of hyperk\"ahler manifolds with $B$ connected, $b_0\in B$
a base point, $X$ the fibre of $\varphi$ over $b_0$, and $2n=\dim_\bbC(X)$. Let $V=H^2(X,\bbQ)$
and denote by $q$ the BBF form on $V$. Then the following statements hold:
\begin{enumerate}
\item There exists a finite \'etale covering $B'\to B$ such that the action of $\pi_1(B',b_0')$ on $H^{\sdot}(X,\bbQ)$
factors through a homomorphism $\rho\colon\pi_1(B',b_0')\to \Spin(V,q)_\bbQ$;
\item There exists a smooth family of compact complex tori $\psi\colon \AA\to B'$ and an
embedding of the variations of Hodge structures
$$
\bar{\nu}_i\colon R^{i+2n}\varphi'_*\bbQ(n)\hrarr R^{i+2d}\psi_*\bbQ(d),
$$
where $i = -2n,\ldots,2n$, $2d = \dim_\bbC(\AA_{b_0'})$ and $\varphi'\colon\XX'=\XX\times_B B'\to B'$;
\item Assume that there exists a monodromy-invariant cohomology class $h \in H^2(X,\bbQ)^{\pi_1(B,b_0)}$
such that $q(h)>0$. Then one can find $B'$ as above such that $\psi$
is a projective family of abelian varieties.
\end{enumerate}
\end{cor}
\begin{proof}
(1) The monodromy action on the cohomology is induced by a homomorphism $\rho\colon\pi_1(B,b_0)\to \mathrm{MC}(X)$.
Using the subgroup $\mathrm{MC}'\subset \mathrm{MC}(X)$ from Proposition \ref{prop_mc}, we get
a finite index subgroup $\rho^{-1}(\mathrm{MC}')\subset \pi_1(B,b_0)$ and the corresponding covering $B'$.
It satisfies the required properties by the definition of $\mathrm{MC}'$.

(2) Denote by $\tilde{B}$ the universal covering of $B$. Let $\HH$ be the variation of
the Kuga-Satake Hodge structures over $\tilde{B}$ that can be fibrewise described 
as in section \ref{subsec_ks}. The fact that the fibrewise construction indeed defines
a variation of Hodge structures was proved e.g. in \cite[section 3.2, in particular Proposition 3.6]{SS}.
By Proposition \ref{prop_mc}, the image of $\pi_1(B',b_0')$ under $\rho$ is contained in an
arithmetic subgroup of $\Spin(V,q)_\bbQ$. This implies that the action of the fundamental group of
$B'$ preserves some lattice $\Lambda\subset \HH$. Taking the quotient of $\HH/\Lambda$
by the action of $\pi_1(B',b_0')$, we get a family of complex tori $\psi\colon \AA\to B'$.
The embeddings of Hodge structures $\nu_i$ from Theorem \ref{thm_ks} are monodromy-equivariant.
This implies that they induce the embeddings $\bar{\nu_i}$ of the corresponding variations
of Hodge structures over $B'$, see \cite[Proposition 3.7]{So1}.

(3) Under our assumptions, the Kuga-Satake Hodge structures admit a $\Spin(V,q)$-invariant polarization,
see e.g \cite{KSV}. We can therefore fix a polarization type for the Kuga-Satake abelian varieties.
Choosing appropriate finite \'etale cover $B'\to B$, we obtain a map from $B'$ to an arithmetic
quotient of the Siegel half-space $\Gamma \backslash \bbH_{2d} $. For a suitable choice of
$\Gamma$, there exists a projective universal family of abelian varieties over $\Gamma \backslash \bbH_{2d} $,
see e.g. \cite[Theorem 8.11]{PS}. We can then construct the family $\psi$ as the pull-back of the universal family.
\end{proof}

\section{The manifolds of ${\rm K3}^{[n]}$, generalized Kummer and OG6 deformation types}\label{sec_main}

In this section, we recall the results of \cite{DM1}, \cite{DM2}, \cite{Xu} and \cite{MRS}.
They provide the necessary geometric input for the proof of Corollaries \ref{cor_hk},
\ref{cor_abs} and \ref{cor_MT}, showing that in each of the
${\rm K3}^{[n]}$, generalized Kummer and OG6 deformation types there exists at least
one variety with abelian motive.

\subsection{Hilbert schemes of points on K3 surfaces and generalized Kummer varieties}\label{sec_k3n}
Let $S$ be a non-singular complex projective surface. We will denote by $S^{[n]}$
the Hilbert scheme of length $n$ subschemes of $S$ and by $S^{(n)}$ the $n$-th symmetric power
of $S$. Recall that $S^{[n]}$ is non-singular and there exists a Hilbert-Chow morphism $\chi\colon S^{[n]}\to S^{(n)}$.

In \cite{DM1}, the natural stratification of $S^{(n)}$ was used to describe the
motive of $S^{[n]}$. We briefly recall the construction.
Let $\nu = (\nu_1,\ldots, \nu_n)$ denote a partition of $n$,
so that $n = \nu_1 + 2\nu_2 + \ldots + n\nu_n$ and $\nu_i\ge 0$. Let $l(\nu) = \sum_i \nu_i$
and denote by $S^{(\nu)}$ the product $\prod_i S^{(\nu_i)}$. Recall
that the points of $S^{(\nu_i)}$ are the $0$-cycles of length $\nu_i$ in $S$.
Consider the morphism $S^{(\nu)} \to S^{(n)}$ that sends a collection of cycles $(x_1,\ldots, x_n)$ to
the cycle $x_1 + 2x_2 + \ldots + n x_n$. Let $Z_\nu$ denote the product
$S^{(\nu)}\times_{S^{(n)}}S^{[n]}$ with the reduced scheme structure. We get
a commutative square:
\begin{equation}\label{eqn_strat}
\begin{tikzcd}[]
Z_\nu \arrow{r}{q_\nu}\arrow{d}{p_\nu} & S^{[n]} \arrow{d}{\chi} \\
S^{(\nu)} \arrow{r} & S^{(n)}
\end{tikzcd}\nonumber
\end{equation}

By construction, $\dim(S^{(\nu)}) = 2l(\nu)$ and one can show that $\dim(Z_\nu) = n + l(\nu)$.
Note that the symmetric powers of $S$ have quotient singularities. Hence
all the natural operations on Chow groups with rational coefficients are well-defined.
Consider the morphisms $q_{\nu*}\circ p_\nu^*\colon \mathrm{CH}^k_\bbQ(S^{(\nu)}) \to \mathrm{CH}^{k +n -l(\nu)}_\bbQ(S^{[n]})$.
It is shown in \cite[Theorem 5.4.1]{DM1} that the sum of these morphisms gives an isomorphism
of Chow groups
$$
\mathrm{CH}^\sdot_\bbQ(S^{[n]}) \simeq \oplus_\nu \mathrm{CH}^{\sdot - n + l(\nu)}_\bbQ(S^{(\nu)}),
$$
where the direct sum is taken over all partitions of $n$. One deduces from this
an isomorphism of motives
$$
\bfM(S^{[n]}) \simeq \oplus_\nu \bfM(S^{(\nu)})(n - l(\nu)),
$$
which actually holds on the level of Chow motives with rational coefficients,
cf. \cite[Theorem 6.2.1]{DM1}. The motive of $S^{(\nu)}$ is by definition
a submotive of $\otimes_i \bfM(S^{\nu_i})$. The latter is abelian if
the motive of $S$ is abelian. It is shown in \cite[Theorem 7.1]{A1} that
for any complex projective K3 surface $S$ the motive of $S$ is abelian,
hence $\bfM(S^{[n]})\in \catMab$ .

%\subsection{Generalized Kummer varieties}\label{sec_kum}
Analogous arguments show that the motive of a generalized Kummer variety
is abelian. Namely, let $A$ be a complex abelian surface.
Consider the Albanese morphism $a\colon A^{[n+1]}\to A$ which
sends an $(n+1)$-tuple of points on $A$ into their sum.
The fibre $K^n\!A = a^{-1}(0)$ is called the generalized Kummer variety.

To describe the motive of $K^n\!A$ one uses the construction
described above, replacing the symmetric power $A^{(n)}$ by the
fibre of the Albanese morphism $A^{(n)} \to A$. The fibres
of the morphisms $A^{(\nu_i)}\to A$ are finite quotients of abelian varieties.
Therefore, repeating the arguments of \cite{DM1} or using more
general results of \cite{DM2}, we find that $\bfM(K^n\!A) \in \catMab$.
This was also shown in \cite[Theorem 1.1]{Xu} using similar methods.

\subsection{Manifolds of OG6 deformation type}\label{sec_og6}

The manifolds of OG6 deformation type were discovered by O'Grady, see \cite{OG2}.
These 6-dimensional manifolds have originally been constructed as desingularizations
of certain moduli spaces of sheaves on abelian surfaces.
To produce one manifold of this deformation type with abelian motive, one can
use a construction from \cite{MRS}.
In \cite[section 6]{MRS} one finds a diagram of the form
$$
X \stackrel{f}{\longleftarrow} Y_2\stackrel{g_2}{\lrarr} Y_1\stackrel{g_1}{\lrarr} Y.
$$
In this diagram: $X$ is an OG6-type manifold; $Y$ is a ${\rm K3}^{[3]}$-type manifold;
$f$ is a surjective generically finite morphism; $g_1$ and $g_2$ are blow-ups
with centres $Z_1$ and $Z_2$, where $Z_1$ is the disjoint union of 256 projective spaces
and $Z_2$ is isomorphic to $({\rm Bl}_{(A\times A^\vee)[2]}(A\times A^\vee))/\pm 1$ for some
abelian surface $A$ (see the proof of \cite[Proposition 6.1]{MRS}). The motives
of $Y$, $Z_1$ and $Z_2$ are abelian, hence the motive of $Y_2$ is also abelian.
By projection formula, the motive of $X$ embeds into the motive of $Y_2$, hence it is
also abelian.

\section{Deformation principle}\label{sec_def}

\subsection{The setting}\label{sec_set}
In this section, we will assume that
%\begin{equation}\label{eqn_pi2}
$$\pi\colon \XX\to B$$
%\end{equation}
is a smooth proper morphism with connected fibres between
complex analytic spaces, the base $B$ is a connected quasi-projective variety
and for any $b\in B$ the fibre $\XX_b$ is projective. Moreover,
we will assume that there exists a line bundle $\LL\in \mathrm{Pic}(\XX)$ and
a dense Zariski-open subset $U\subset B$ such that $\LL|_{\XX_b}$ is ample for 
any $b\in U$. Let $\XX_U = \pi^{-1}(U)$. We will assume that $\XX_U$ is
a quasi-projective variety, and that the bundle $\LL|_{\XX_U}$ and the
morphism $\pi|_{\XX_U}$ are algebraic.

Let us emphasize that we do not assume the total space $\XX$
to be an algebraic variety, and for the families that we consider below
(see Proposition \ref{prop_curve}) it will typically not be algebraic.
% We will consider the local systems $\HH^{2p} = R^{2p}\pi_*\bbC$ on $B$.

%\begin{rem}
%Although we assume all the fibres $\XX_b$ to be projective, we do not
%require $\LL_{\XX_b}$ to define a (quasi-)polarization on $\XX_b$ for $b\in B\setminus U$.
%\end{rem}

\subsection{The statement and preliminary constructions}

The following proposition is a version of the deformation principle for
motivated cohomology classes, see section \ref{sec_mot} and \cite[Theorem 0.5]{A1}.
We remark that similar results about specialization of motivated cycles
were obtained in \cite[Corollary 4.7]{Ca}. We can not apply those results in
our setting, because we do not assume the total space $\XX$ of our family to
be an algebraic variety.

\begin{prop}\label{prop_def}
In the above setting \ref{sec_set}, let $b_0, b_1\in B$ be two points and consider
a section $\xi \in H^0(B,R^{2p}\pi_*\bbC)$.
Let $\xi_{b_i}\in H^{2p}(\XX_{b_i}, \bbC)$, $i=0,1$ be the values of $\xi$ at $b_i$. Then $\xi_{b_0}$ is 
motivated if and only if $\xi_{b_1}$ is motivated.
\end{prop}

The proof is given below in \ref{sec_def_proof}.
It uses the same idea as in \cite{A1},
but in our case the variety $\XX$ is not algebraic, so
we start by explaining some preliminary constructions.

First note that since $B$ is quasi-projective, we can connect any
two points in it by a chain of integral curves. Since $U$ is dense
in $B$, we can also make sure that each of those curves intersects $U$.
Choosing intermediate points between $b_0$ and $b_1$ at the
intersections of the curves in the chain, we reduce to the case
when $B$ is an integral quasi-projective curve. We may pull back
the family $\XX$ to the normalization of this curve, and
then assume that $B$ is smooth.

The case when both $b_0$ and $b_1$ lie in $U$ reduces to \cite[Theorem 0.5]{A1}, since in this case
we can shrink $B$ and assume that $\LL$ is $\pi$-ample. If both $b_0$ and $b_1$
lie in $B\setminus U$, we can choose an intermediate point $b_3\in U$,
and hence we reduce to the case when one of the points
from the statement of Proposition \ref{prop_def} lies in $U$
and the other in $B\setminus U$. To fix the notation, we
will assume that $b_0\in U$, $b_1\in B\setminus U$. 
After shrinking $B$, we may moreover assume that  $U = B\setminus \{b_1\}$.
%and the morphism $\pi$ is projective over $U$. Let $\XX_U = \pi^{-1}(U)$.
We will denote by $X_i = \pi^{-1}(b_i)$ the two fibres.

The curve $B$ is quasi-projective, hence there exists a
projective curve $\bar{B}$ that contains $B$ as a Zariski-open subset.
Let us denote the boundary $\bar{B}\setminus B =\{z_1,\ldots,z_n\}$,
and let $\bar{U} = \bar{B}\setminus\{b_1\}$. Our next goal
is to find a suitable compactification of $\XX$. We use the fact
that over a punctured neighbourhood of every point $z_i$ the
morphism $\pi$ is projective and we can extend it over the puncture.
The details appear in the following lemma.

\begin{lem}\label{lem_xbar}
There exists a compact complex manifold $\bar{\XX}$, a flat morphism $\bar{\pi}\colon\bar{\XX}\to\bar{B}$,
a line bundle $\bar{\LL}\in\mathrm{Pic}(\bar{\XX})$ and an open embedding $j\colon \XX \hrarr \bar{\XX}$
that satisfy the following conditions:
\begin{enumerate}
\item $\bar{\pi}|_{j(\XX)} = \pi$ and $\bar{\XX}\setminus j(\XX) = \cup_{i=1}^n \bar{\pi}^{-1}(z_i)$;
\item The open subset $\bar{\XX}_{\bar{U}} = \bar{\pi}^{-1}(\bar{U}) \subset \bar{\XX}$ is a quasi-projective variety;
\item The restriction of $\bar{\pi}$ to $\bar{\XX}_{\bar{U}}$ is an algebraic morphism onto $\bar{U}$ and
$\bar{\LL}|_{\bar{\XX}_{\bar{U}}}$ is $\bar{\pi}$-ample algebraic bundle.
\end{enumerate}
\end{lem}
\begin{proof}
Recall that $\pi_U\colon \XX_U\to U$ is a smooth morphism
of quasi-projective varieties with relatively ample line bundle $\LL_U = \LL|_{\XX_U}$.
For a big enough integer $k$, consider the vector bundle $E = \pi_{U*}(\LL^k_U)$ over $U$.
We can find such $k$ that the canonical morphism $\pi_U^*(E)\to \LL_U$ induces a closed embedding
$\XX_U\hrarr \bbP(E)$. Note that $E$ is an algebraic vector bundle over the quasi-projective
curve $U$, so we can extend it to a vector bundle $\bar{E}$ over $\bar{U}$.
Denote by $\ZZ$ the closure of $\XX_U$ in $\bbP(\bar{E})$. Then $\ZZ$ is a quasi-projective
variety fibred over $\bar{U}$. It may be singular, the singularities lying over the points $z_i$.
By a theorem of Hironaka, there exists a resolution of singularities $\rho\colon \hat{\ZZ}\to \ZZ$.
This means that $\hat{\ZZ}$ is a non-singular quasi-projective variety and the morphism
$\rho$ is a composition of blow-ups with centres lying over the singular locus of $\ZZ$,
see e.g. \cite[Definition 7.1 and Theorem 7.5]{GPR}. In our case the blow-ups occur in
the fibres over the points $z_i$, and $\rho$ is an isomorphism over $\XX_U$.
Hence $\hat{\ZZ}$ contains $\XX_U$ as an open subset, and the morphism
$\hat{\ZZ}\to \bar{U}$ agrees with $\pi_U$ on this subset. We get the manifold $\bar{\XX}$
by gluing $\hat{\ZZ}$ with $\XX$ along their common open subset $\XX_U$.
Then $\XX$ is embedded into $\bar{\XX}$ by construction, and we get a morphism $\bar{\pi}\colon\bar{\XX}\to \bar{B}$
that has the claimed property (1). The morphism $\bar{\pi}$ is flat, because it is equidimensional and $\bar{\XX}$, $\bar{U}$
are smooth manifolds (see \cite[page 114]{GPR}).

The open subset $\bar{\XX}_{\bar{U}}$ is by construction identified with $\hat{\ZZ}$,
hence it is quasi-projective and (2) is satisfied. The first part of (3) is
also satisfied because $\hat{\ZZ}$ is a blow-up of the algebraic variety $\ZZ$.

Next we prove the existence of the line bundle $\bar{\LL}$.
Consider the line bundle $\LL' = \OO_{\bbP(\bar{E})/\bar{U}}(1)|_\ZZ$.
The bundle $\LL'$ is relatively ample over $\bar{U}$ and $\LL'|_{\XX_U}\simeq \LL_U^k$.
The blow-up morphism $\rho$ is projective, and if we denote by $E_j$ the exceptional
divisors of $\rho$, the bundle $\LL'' = \rho^*\LL'(\sum a_j E_j)$ is relatively ample over $\bar{U}$ for
suitably chosen integers $a_j$. We have $\LL''|_{\XX_U} \simeq \LL_U^k$ and
$\LL^k|_{X_U} \simeq \LL_U^k$, and we define $\bar{\LL}$ by gluing $\LL''$
and $\LL^k$ over the open subset $\XX_U$. The second part of (3) is then satisfied because
$\bar{\LL}|_{\bar{\XX}_{\bar{U}}} \simeq \LL''$ is algebraic and relatively ample over $\bar{U}$.
\end{proof}

In what follows we will implicitly identify $\XX$ with its image in $\bar{\XX}$
under the embedding $j$ from the lemma above.

\begin{lem}\label{lem_moish}
The manifold $\bar{\XX}$ from Lemma \ref{lem_xbar} is Moishezon.
There exists a projective manifold $\hat{\XX}$ and a birational
morphism $r\colon \hat{\XX}\to \bar{\XX}$ that is an isomorphism over $\bar{\XX}\setminus X_1$
and such that $r^{-1}(X_1)$ is a simple normal crossing divisor.
\end{lem}
\begin{proof}
We use the notation from the statement of Lemma \ref{lem_xbar}.
We refer to \cite[chapter VII, \S 6]{GPR} for the discussion of Moishezon manifolds.
Our proof is standard: we use the line bundle $\bar{\LL}$ (or rather its power) to produce
a birational map from $\bar{\XX}$ to a projective variety.

Let $\bar{\pi}_{\bar{U}} = \bar{\pi}|_{\bar{\XX}_{\bar{U}}}$
and $E = \bar{\pi}_{\bar{U}*}( \bar{\LL}|_{\bar{\XX}_{\bar{U}}} )$.
After possibly replacing the bundle $\bar{\LL}$ by its power, we may assume that the
canonical morphism $\bar{\pi}_{\bar{U}}^*(E) \to \bar{\LL}|_{\bar{\XX}_{\bar{U}}}$
defines a closed embedding $\bar{\XX}_{\bar{U}}\hrarr \bbP(E)$. Now consider the
vector bundle $E' = \bar{\pi}_*(\bar{\LL})$ over $\bar{B}$ and note that
$E'|_{\bar{U}} \simeq E$. Since we do not assume that $\bar{\LL}|_{X_1}$ is ample
(or even globally generated) the morphism $\bar{\pi}^*(E')\to \bar{\LL}$ only
defines a meromorphic map $\varphi\colon \bar{\XX} \dashrightarrow \bbP(E')$
whose restriction to $\bar{\XX}_{\bar{U}}$ coincides with the embedding
$\bar{\XX}_{\bar{U}}\hrarr \bbP(E)$. The indeterminacy locus of $\varphi$
is contained in the fibre $X_1$. We can resolve the indeterminacy of $\varphi$
by a sequence of blow-ups $r'\colon \tilde{\XX}\to \bar{\XX}$, so that
$\varphi$ lifts to a morphism $\tilde{\varphi}\colon\tilde{\XX}\to \bbP(E')$.
The latter morphism is proper, hence its image is an analytic subvariety
that we denote by $\YY$. Since $\YY$ is a subvariety of the projective manifold
$\bbP(E')$, it is also projective. This shows that $\bar{\XX}$ is Moishezon.

Next we show that there is a birational morphism $r''\colon \hat{\YY}\to\tilde{\XX}$
for some projective manifold $\hat{\YY}$. We consider the birational map
$\tilde{\varphi}^{-1}\colon \YY\dashrightarrow \tilde{\XX}$. We resolve the indeterminacy
of the latter map and the singularities of $\YY$ by a sequence of blow-ups
$s\colon \hat{\YY}\to \YY$. Then $\hat{\YY}$ is a projective manifold and there is
a birational morphism $r''\colon \hat{\YY}\to\tilde{\XX}$. Note that the blow-ups occur
only in the fibre over the point $b_1\in B$, so all the constructed birational
morphisms are isomorphisms over $\bar{\XX}_{\bar{U}}$. The preimage of $b_1$ in
$\hat{\YY}$ is a divisor. By blowing up $\hat{\YY}$ further, we make sure that
this divisor has simple normal crossings. We define $\hat{\XX}$ to be the result
of this final sequence of blow-ups and $r$ to be the induced map to $\bar{\XX}$.
By construction, $\bar{\XX}$ has the claimed properties.
\end{proof}

Consider the morphism $r\colon \hat{\XX}\to\bar{\XX}$ constructed in Lemma \ref{lem_moish}.
Then $D = (\bar{\pi}\circ r)^{-1}(b_1)$ is a simple normal crossing divisor by the lemma,
and $r(D) = X_1$. Let us denote by $\hat{X}_1$ the irreducible component of $D$
that dominates $X_1$. It is a non-singular projective variety. The restriction of $r$
to $\hat{X}_1$ is a birational morphism $r_1\colon \hat{X}_1\to X_1$.
We denote by $\alpha_i\colon X_i \hookrightarrow \bar{\XX}$ the closed immersions.
Since $r$ is an isomorphism over $\bar{\XX}\setminus X_1$, the immersion
$\alpha_0$ lifts to $\hat{\alpha}_0\colon X_0 \hookrightarrow \hat{\XX}$.
We summarize our constructions in a diagram:
\begin{equation}\label{eqn_diagram}
\begin{tikzcd}
X_0 \arrow[equal]{d}\arrow[bend left, hook]{rr}{\hat{\alpha}_0}\arrow[hook]{r} & \XX_U \arrow[hook]{d}\arrow[hook]{r}
& \hat{\XX} \arrow{d}{r}\arrow[hookleftarrow]{r}{\hat{\alpha}_1} & \hat{X}_1\arrow{d}{r_1} \\
X_0 \arrow[hook]{r}\arrow[bend right, hook]{rr}{\alpha_0} & \XX \arrow[hook]{r} & \bar{\XX} \arrow[hookleftarrow]{r}{\alpha_1} & X_1
\end{tikzcd}
\end{equation}

The projective manifold $\hat{\XX}$ can be used as the 0-th term of a simplicial
resolution for $\bar{\XX}$, which is an augmented simplicial variety
$\mathcal{S}_\sdot \to \bar{\XX}$ with $\mathcal{S}_i$ smooth projective and $\mathcal{S}_0 = \hat{\XX}$.
Such resolution can be constructed as in \cite[6.2.5]{D2}. Note that both
$\bar{\XX}$ and $\hat{\XX}$ are smooth manifolds of the same dimension, hence by
projection formula \cite[IX.7.3]{Iv} the pull-back on cohomology $r^*$ is injective.
Thus we have the following exact sequence (cf. \cite[Proposition 8.2.5]{D2}):
$$
0\lrarr H^{2p}(\bar{\XX}, \bbC) \stackrel{r^*}{\lrarr} H^{2p}(\hat{\XX},\bbC) \stackrel{\delta_0^*-\delta_1^*}{\lrarr} H^{2p}(\mathcal{S}_1,\bbC),
$$
where $\delta_0,\delta_1\colon \mathcal{S}_1 \rightrightarrows \mathcal{S}_0 = \hat{\XX}$ are the face maps.

\begin{defn}\label{defn_mot1}
Define the motive ${\bf H}^{2p}(\bar{\XX})$ to be the kernel of the morphism
$$\delta_0^*-\delta_1^*\colon {\bf H}^{2p}(\hat{\XX}) \to {\bf H}^{2p}(\mathcal{S}_1).$$
\end{defn}

Let $G_i = \pi_1(B,b_i)$. There exist canonical isomorphisms $H^0(B,R^{2p}\pi_*\bbC) \simeq H^{2p}(X_i,\bbC)^{G_i}$.

\begin{lem}\label{lem_kernels}
We have the following equalities:
\begin{equation}\label{eqn_ker}
\ker(\alpha_0^*\colon H^{2p}(\bar{\XX},\bbC) \to H^{2p}(X_0,\bbC)) = \ker(\alpha_1^*\colon H^{2p}(\bar{\XX},\bbC) \to H^{2p}(X_1,\bbC)),
\end{equation}
and
\begin{equation}\label{eqn_im}
\im(\alpha_i^*\colon H^{2p}(\bar{\XX},\bbC) \to H^{2p}(X_i,\bbC)) = H^{2p}(X_i,\bbC)^{G_i}, \quad i=0,1.
\end{equation}
\end{lem}

\begin{proof}
The kernels in (\ref{eqn_ker}) can be identified with the kernel of
the composition
\begin{equation}\label{eqn_comp}
H^{2p}(\bar{\XX},\bbC) \to H^{2p}(\XX,\bbC) \to H^0(B,R^{2p}\pi_*\bbC),
\end{equation}
where the last morphism comes from the Leray spectral sequence. This proves (\ref{eqn_ker}).

To prove (\ref{eqn_im}), it is sufficient to check that $H^{2p}(\bar{\XX},\bbC) \to H^{2p}(X_0,\bbC)^{G_0}$
is surjective, because this would imply surjectivity of the composition (\ref{eqn_comp}). Since
$X_0 \subset \XX_U$, we have the composition
\begin{equation}\label{eqn_comp2}
H^{2p}(\bar{\XX},\bbC) \stackrel{\beta}{\to} H^{2p}(\XX_U,\bbC) \stackrel{\gamma}{\to} H^{2p}(X_0,\bbC)^{G_0}.
\end{equation}

The line bundle $\LL$ defines a polarization on the fibres over $U$, hence the Leray spectral sequence
degenerates at $E_2$. Since the action of $\pi_1(U,b_0)$ on the cohomology of $X_0$ factors through $G_0$,
it follows that the morphism $\gamma$ is surjective. The cohomology of $\XX_U$ carries a mixed Hodge
structure, and since the Hodge structure on $H^{2p}(X_0,\bbC)$ is pure, we deduce that
the restriction of $\gamma$ to $W_{2p}H^{2p}(\XX_U,\bbC)$ is still surjective. On the other hand,
$\bar{\XX}$ is a smooth Moishezon compactification of $\XX_U$, and \cite[Corollaire (3.2.17)]{D1}
shows that $W_{2p}H^{2p}(\XX_U,\bbC)$ is the image of $\beta$. This completes the proof of (\ref{eqn_im}).
\end{proof}

\begin{defn}\label{defn_mot2}
Define the following motives:
$$
{\bf H}^{2p}(X_0)^{G_0} = \im ( \hat{\alpha}^*_0\colon {\bf H}^{2p}(\bar{\XX}) \to {\bf H}^{2p}(X_0)),
$$
$$
{\bf H}^{2p}(X_1)^{G_1} = \im ( r_{1*}\circ \hat{\alpha}^*_1\colon {\bf H}^{2p}(\bar{\XX}) \to {\bf H}^{2p}(X_1)).
$$
\end{defn}

Note that $\hat{\alpha}_0^*\circ r^* = \alpha_0^*$ and by projection formula $r_{1*}\circ \hat{\alpha}_1^*\circ r^* =
r_{1*}\circ r_1^*\circ\alpha_1^* = \alpha_1^*$, see the diagram (\ref{eqn_diagram}). Therefore, it follows from (\ref{eqn_im}) that
${\bf H}^{2p}(X_i)^{G_i}$ are the motives representing $H^{2p}(X_i,\bbC)^{G_i}$.

\subsection{Proof of Proposition \ref{prop_def}}\label{sec_def_proof}

Consider the motives $\bfH^{2p}(X_i)^{G_i}$ introduced in Definition \ref{defn_mot2}.
The two quotient maps ${\bf H}^{2p}(\bar{\XX}) \to {\bf H}^{2p}(X_i)^{G_i}$ have the same kernel by (\ref{eqn_ker}).
Hence the quotients ${\bf H}^{2p}(X_i)^{G_i}$ are canonically isomorphic, and the corresponding
isomorphism in cohomology is induced by $H^{2p}(X_i,\bbC)^{G_i} \simeq H^{2p}(B,R^{2p}\pi_*\bbC)$,
by construction.
We conclude that the isomorphism $H^{2p}(X_0,\bbC)^{G_0}\simeq H^{2p}(X_1,\bbC)^{G_1}$ lifts to
the category of Andr\'e motives. It follows that for any section $\xi\in H^0(B, R^{2p}\pi_*\bbC)$ the cohomology
class $\xi_{b_0}\in H^{2p}(X_0,\bbC)$ is motivated if and only if $\xi_{b_1}\in H^{2p}(X_1,\bbC)$ is motivated.
This finishes the proof.

\section{Motives of hyperk\"ahler manifolds}\label{sec_mot_hk}

\subsection{Moduli theory of compact hyperk\"ahler manifolds}\label{sec_moduli}
In what follows, we consider hyperk\"ahler manifolds of a fixed deformation type.
We start by recalling some necessary facts about moduli spaces of such manifolds.

First recall that two compact hyperk\"ahler
manifolds $X_1$ and $X_2$ are called deformation equivalent if there exists
a smooth analytic family $\pi\colon\XX\to B$ over a connected complex analytic
base $B$ such that all fibres are compact hyperk\"ahler manifolds and
there exist two fibres isomorphic to $X_1$ and $X_2$. In particular,
the underlying topological manifolds are diffeomorphic, and one can
show that properly normalized Beauville-Bogomolov-Fujiki forms on $X_1$
and $X_2$ are equal, see e.g. \cite[section 2.2]{So3}. Hence we can fix a
lattice $\Lambda$ representing the second integral cohomology
with the BBF form for our deformation equivalence class.

We will formulate the moduli theory in the language of marked hyperk\"ahler manifolds,
following \cite{Hu2}.
Recall that a marked hyperk\"ahler manifold is a pair $(X,\varphi)$,
where $X$ is a hyperk\"ahler manifold of the fixed deformation type
and $\varphi\colon H^2(X,\bbZ)\stackrel{\sim}{\to} \Lambda$ is an isomorphism
of lattices which is called a marking. If $\sigma\in H^0(X,\Omega_X^2)$ is the symplectic form, then
$q(\sigma) = 0$ and $q(\sigma,\bar{\sigma}) > 0$, where $q$ is the BBF form.
We denote by $\DD\subset \bbP(\Lambda\otimes\bbC)$ the period domain defined as
$\DD = \{x\in\bbP(\Lambda\otimes\bbC) \st q(x) = 0,\, q(x,\bar{x}) > 0\}$, which is an open subset
of the quadric in $\bbP(\Lambda\otimes\bbC)$ defined by the BBF form.
We define the period of a marked hyperk\"ahler manifold $(X,\varphi)$ to
be the point $\rho(X,\varphi) = [\varphi(\sigma)]\in \DD$.

We call two marked hyperk\"ahler manifolds $(X_1,\varphi_1)$ and $(X_2,\varphi_2)$
isomorphic if there exists a biholomorphic isomorphism $f\colon X_1\stackrel{\sim}{\to}X_2$
such that $\varphi_2 = \varphi_1\circ f^*$. In this case clearly $\rho(X_1,\varphi_1) = \rho(X_2,\varphi_2)$,
since $f^*$ is an isomorphism of Hodge structures. We let $\mathfrak{M}$ be the set
of isomorphism classes of marked hyperk\"ahler manifolds of the fixed deformation type.
Then $\rho$ is a well-defined map from $\mathfrak{M}$ to $\DD$ called the period map.
It is known (see e.g. \cite[Proposition 4.3]{Hu2}) that one can endow $\mathfrak{M}$
with the structure of a complex manifold in such a way that $\rho$ becomes holomorphic.
We briefly recall how to define the topology and complex analytic charts on $\mathfrak{M}$.
Let $(X,\varphi)\in \mathfrak{M}$ and consider the universal deformation $\pi\colon \XX\to \Delta^k$ of the
complex manifold $X$. Here $\Delta$ is the unit disc, $k = \dim H^1(X,T_X)$, and the fibre of
$\pi$ over zero is isomorphic to $X$. The base of the universal deformation is smooth
by the well known result of Bogomolov-Tian-Todorov. The total space $\XX$ is diffeomorphic
to the product $X\times \Delta$, and we can canonically identify $H^2(X_t,\bbZ)$ with $H^2(X_0,\bbZ) = H^2(X,\bbZ)$,
where $t\in \Delta^k$ and $X_t=\pi^{-1}(t)$. Composing with $\varphi$, we get
a marking $\varphi_t$ on $X_t$ for every $t\in \Delta^k$, hence a map $\mu\colon \Delta^k\to \mathfrak{M}$.
By the local Torelli theorem (see \cite{Be}) the differential at $0\in \Delta^k$
of the composition $\rho\circ\mu$ is an isomorphism, in particular $\rho\circ\mu$
maps some smaller polydisc $U\subset \Delta^k$ biholomorphically onto
an open subset of $\DD$. This implies that $\mu|_{U}$ is an embedding.
We endow $\mathfrak{M}$ with the finest topology for which all such
embeddings are continuous. We use the images of polydiscs under the described
embeddings as a system of complex analytic charts on $\mathfrak{M}$.

The constructed complex manifold $\mathfrak{M}$ is usually non-Hausdorff.
Recall that two points $x,y\in \mathfrak{M}$ are called inseparable if
for any non-empty open neighbourhoods $U_x$ of $x$ and $U_y$ of $y$ we have
$U_x\cap U_y \neq \emptyset$. The inseparability is not an equivalence
relation for points in arbitrary topological spaces, but it is shown
in \cite[section 4.3]{Hu2} that for $\mathfrak{M}$ it actually is.
So we can define the Hausdorff reduction $\overline{\mathfrak{M}}$ to be
the set of equivalence classes of inseparable points in $\mathfrak{M}$.
One checks that $\overline{\mathfrak{M}}$ is a Hausdorff complex manifold,
and that the period map $\rho$ factors through $\overline{\rho}\colon \overline{\mathfrak{M}}\to \DD$.

Let $\MM$ be a connected component of $\mathfrak{M}$ and let $\overline{\MM}$
be its Hausdorff reduction. One of the central results of the moduli
theory of compact hyperk\"ahler manifolds can be formulated as follows,
see \cite[Theorem 4.29]{Ve2}, \cite[Corollary 5.9]{Hu2}:

\begin{thm}[The global Torelli theorem]\label{thm_torelli}
The period map $\overline{\rho}\colon \overline{\MM} \to \DD$ is a biholomorphic isomorphism.
\end{thm}

In what follows, we will always fix a connected component $\MM$ of the moduli space
as above. Let $(X_1,\varphi_1)\in\MM$ and $X_2$ be a hyperk\"ahler manifold deformation
equivalent to $X_1$. Then $X_1$ and $X_2$ are two fibres of a smooth family $\pi\colon\XX\to B$
over a connected base. Pulling back the family $\XX$ to the universal covering $\tilde{B}$
of the base, we get a family $\tilde{\pi}\colon\tilde{\XX}\to \tilde{B}$. The space $\tilde{B}$
being simply connected, the local system $R^2\tilde{\pi}_*\bbZ$ is trivial and by parallel
transport we identify its fibres with $H^2(X_1,\bbZ)$. Composing with $\varphi_1$,
we get markings on all fibres of the map $\tilde{\pi}$, in particular on $X_2$.
We denote the latter marking by $\varphi_2$. Thus $\tilde{\XX}$ becomes a family
of marked hyperk\"ahler manifolds, and we get a holomorphic map from $\tilde{B}$
to $\mathfrak{M}$. Since $\tilde{B}$ is connected, the image of this map is
contained in $\MM$. In particular we observe that $(X_2,\varphi_2)\in\MM$.
Hence it is possible to find a marking for $X_2$, so that the corresponding marked
manifold defines a point in the same connected component $\MM$. We will use
this observation later.

Given a non-zero element $h\in \Lambda$, we denote by $\Lambda_h\subset \Lambda$
the orthogonal complement to $h$ and let $\DD_h= \DD\cap \bbP(\Lambda_h\otimes\bbC)$, $\MM_h = \rho^{-1}(\DD_h)$.
Recall the projectivity criterion for a hyperk\"ahler manifold $X$: by \cite[Theorem 3.11]{Hu1}
$X$ is projective if and only if there exists a class $a\in H^2(X,\bbZ)\cap H^{1,1}(X)$ such that $q(a) > 0$.
Equivalently, $X$ is projective if and only if there exists a line bundle $L\in\mathrm{Pic}(X)$ with $q(c_1(L)) > 0$
(although it is not always true that this line bundle is ample).
It follows that for $h\in \Lambda$ with $q(h) > 0$ all hyperk\"ahler manifolds parametrized by $\MM_h$
are projective.

Finally, let us recall some necessary facts about inseparable points in $\MM$. Consider a point $(X,\varphi)\in \MM$.
%and $p = \rho(X,\varphi)\in \DD$, where $\rho\colon\MM\to \DD$ is the period map as before.
If another point $(X',\varphi')\in \MM$ is inseparable from $(X,\varphi)$, it is known from \cite[Theorem 4.3]{Hu1}
that $X$ and $X'$ are bimeromorphic.
%Since all marked manifolds inseparable from $(X,\varphi)$ lie over the same point $p$ of
%the period domain, their second cohomology groups are identified via the markings.
It is possible to distinguish between different bimeromorphic
models of $X$ using their K\"ahler cones. Namely, we consider the open set $\{x\in H^{1,1}(X,\bbR)\st q(x)>0\}$
which has two connected components, because the signature of $q|_{H^{1,1}(X,\bbR)}$ is $(1,h^{1,1}(X)-1)$.
We define the positive cone $\CC_X$ to be
the connected component of this set that contains the K\"ahler cone $\mathcal{K}_X$,
the latter being the set of all cohomology classes represented by K\"ahler forms on $X$.
If $f\colon X\dashrightarrow X'$ is a bimeromorphic map, then $f^*\colon H^2(X',\bbR)\to H^2(X,\bbR)$
is well-defined, see \cite[Lemma 2.6]{Hu1}. We define the birational K\"ahler cone $\mathcal{BK}_X\subset \CC_X$
to be the union of all preimages $f^*(\mathcal{K}_{X'})$ for all bimeromorphic models $X'$ of $X$.
These preimages are the connected component of $\mathcal{BK}_X$, and $\mathcal{K}_X$ is one of them.

Below we will need a way to make sure that $\MM$ contains no points inseparable from $(X,\varphi)$.
To do this one has to show that $\mathcal{K}_X = \CC_X$, see \cite[Theorem 2.2(4)]{Ma}, because
then $\mathcal{BK}_X = \mathcal{K}_X$ and $X$ has only one bimeromorphic model. To understand when
$\mathcal{K}_X = \CC_X$, one can use the description of the boundary of $\mathcal{K}_X$ given
in \cite[Theorem 1.19]{AV}. It is shown in loc. cit. that when $\mathcal{K}_X$ is strictly
contained in $\CC_X$, then its boundary contains a face that is cut out by a hyperplane
orthogonal to a certain cohomology class $z\in H^2(X,\bbZ)\cap H^{1,1}(X)$ called an MBM class,
see \cite[Definition 1.13]{AV}. Such a class must have negative BBF square because of
the signature of $q$ mentioned above. Hence to make sure that $\mathcal{K}_X = \CC_X$
it suffices to show that there exist no negative classes in $H^2(X,\bbZ)\cap H^{1,1}(X)$.
This is the case e.g. when $\mathrm{Pic}(X) = 0$ or when $\mathrm{Pic}(X)$ is spanned
by a line bundle with first Chern class of positive BBF square, cf. \cite[Theorem 2.2(5)]{Ma}.
In particular, if $h\in \Lambda$, $q(h)>0$ and $\DD_h$ is the divisor in the period domain
introduced above, then for a very general point $p\in \DD_h$ the Picard group
of a marked manifold with period $p$ has rank one. The generator of the Picard
group of such a manifold has positive BBF square, hence the condition discussed above
is satisfied. We conclude that all points of $\MM_h$ lying over very general points of $\DD_h$
are Hausdorff. A more detailed discussion
of inseparable points can be found in \cite[section 5.3]{Ma}, see in particular
\cite[Theorem 5.16]{Ma}.

\subsection{Constructing smooth families of hyperk\"ahler manifolds}
We will need the following lemma about filling in families of hyperk\"ahler
manifolds over the punctured disc. We denote the unit disc by $\Delta$ and
the punctured disc by $\Delta^*$. If $\pi'\colon \XX'\to \Delta^*$ is a
family of hyperk\"ahler manifolds, its marking is an isomorphism
$\varphi'\colon R^2\pi'_*\bbZ \simeq \underline{\Lambda}$, where $\underline{\Lambda}$
is the constant sheaf with fibre $\Lambda$. In particular,
the monodromy action on $H^2$ is trivial for a marked family.
The marking induces a period map $\gamma'\colon \Delta^*\to \DD$.
We will say that some condition is satisfied for a very general point $t\in \Delta$,
if there exists a countable subset $Z\subset \Delta$ such that the
condition is satisfied for any $t\in \Delta\setminus Z$.

\begin{lem}\label{lem_mod}
Let $\pi\colon \XX\to \Delta$ be a flat projective morphism and
$\pi'\colon \XX'\to \Delta^*$ its restriction to $\Delta^*$.
Assume that $\pi'$ is a smooth family of marked hyperk\"ahler manifolds,
and the period map $\gamma'\colon \Delta^*\to \DD$ extends
to a morphism  $\gamma\colon \Delta\to \DD$. Assume that a very general
fibre of $\pi$ has Picard rank one. Let $(X,\varphi)$
be a marked hyperk\"ahler manifold such that $\rho(X,\varphi) = \gamma(0)$.
Then there exists a finite ramified cover $\alpha\colon \Delta\to \Delta$,
and a smooth family of hyperk\"ahler manifolds $\tilde{\pi}\colon \tilde{\XX}\to \Delta$
such that $\tilde{\XX}_0 \simeq X$ and $\alpha^*\XX|_{\Delta^*} \simeq \tilde{\XX}|_{\Delta^*}$,
after possibly shrinking $\Delta$.
Any line bundle $\LL'\in\mathrm{Pic}(\tilde{\XX}|_{\Delta^*})$
can be extended to a line bundle $\LL\in\mathrm{Pic}(\tilde{\XX})$.
\end{lem}
\begin{proof}
The statement is essentially equivalent to \cite[Theorem 0.8]{KLSV}.
Following the argument from \cite[section 3]{KLSV}, we pull back
the universal family of $X$ via $\gamma$ and obtain a smooth
family of hyperk\"ahler manifolds $\xi\colon \YY\to \Delta$
with central fibre $X$. There exists a finite covering $\alpha\colon \Delta\to \Delta$,
and a cycle $\ZZ\subset \alpha^*\XX\times_\Delta\alpha^*\YY$ that induces a
birational isomorphism between $\alpha^*\XX$ and $\alpha^*\YY$ over $\Delta$.
We define $\tilde{\XX} = \alpha^*\YY$.

The local system $R^2\xi_*\bbZ$ is trivial and by parallel transport we
identify its fibres with $H^2(\YY_0,\bbZ)\simeq H^2(X,\bbZ)$. Using the marking $\varphi$
we then obtain an isomorphism $R^2\xi_*\bbZ\simeq \underline{\Lambda}$, i.e. $\YY$
becomes a family of marked hyperk\"ahler manifolds. For $t\in \Delta^*$ let us
denote by $(\XX_t,\varphi_t)$ and $(\YY_t,\varphi'_t)$ the fibres of the two
families $\XX$ and $\YY$ with the induced markings.
Thus we obtain two maps from $\Delta^*$ to $\MM$ whose compositions with
the period map $\rho$ are equal, by construction. Recall that we denote by $\MM$
one connected component of the moduli space of marked hyperk\"ahler manifolds,
and that by the global Torelli Theorem \ref{thm_torelli} $\rho$ induces
an isomorphism between the Hausdorff reduction of $\MM$ and $\DD$.
It follows that for $t\in \Delta^*$ the marked manifolds $(\XX_t,\varphi_t)$ and $(\YY_t,\varphi'_t)$
represent either the same point in $\MM$ or a pair of inseparable points.
By our assumption, for a very general $t\in \Delta^*$ the Picard group
of $\XX_t$ is generated by an ample line bundle, hence the K\"ahler cone of $\XX_t$ coincides with a
connected component of the positive cone. This implies that $\MM$ contains no
inseparable points over $\gamma(t)$, see the discussion at the end of section \ref{sec_moduli}
or \cite[section 5.3 and Theorem 5.16]{Ma}.
Hence the cycle $\ZZ_s$ is the graph of an isomorphism between $\XX_{\alpha(s)}$ and $\YY_{\alpha(s)}$ for a very general $s\in \Delta$.
The subset of $s\in \Delta$ for which $\ZZ_s$ defines an isomorphism of the fibres is Zariski-open.
This implies that the families $\alpha^*\XX|_{\Delta^*}$ and $\tilde{\XX}|_{\Delta^*}$
are isomorphic, after possibly shrinking $\Delta$.

To prove the last claim of the lemma, note that we have the following isomorphism
$$\mathrm{Pic}(\tilde{\XX})\simeq \mathrm{ker}(H^0(\Delta, R^2\tilde{\pi}_*\bbZ)\to H^0(\Delta, R^2\tilde{\pi}_*\OO_{\tilde{\XX}})).$$
This isomorphism follows from the exponential exact sequence using the fact that $\Delta$
is a Stein manifold and the fibres of $\tilde{\pi}$ are simply connected.
The local system $(R^2\tilde{\pi}_*\bbZ)|_{\Delta^*}$ is trivial, and its section that defines $\LL'$ extends to a section
of $R^2\tilde{\pi}_*\bbZ$. This extension still lies in the kernel on the right hand
side of the above formula, because the sheaf $R^2\tilde{\pi}_*\OO_{\tilde{\XX}}$ is locally free.
Thus we get a line bundle $\LL\in \mathrm{Pic}(\tilde{\XX})$ that extends $\LL'$.
\end{proof}

As we recalled above, given a hyperk\"ahler manifold $X_1$,
we can choose a marking $\varphi_1\colon H^2(X_1,\bbZ) \stackrel{\sim}{\to} \Lambda$
such that $(X_1,\varphi_1) \in \MM$. Assume that $X_1$ is projective with a very ample line
bundle $L$, and let $h = \varphi_1(c_1(L))$. Then $(X_1,\varphi_1) \in \MM_h$.

\begin{prop}\label{prop_curve}
In the above setting, assume that we have another marked hyperk\"ahler manifold
$(X_2,\varphi_2) \in \MM_h$. Then there exists a connected quasi-projective curve $C$,
a smooth analytic family of hyperk\"ahler manifolds $\pi\colon \XX\to C$ and two points $x_1,x_2\in C$
such that  $X_i\simeq \pi^{-1}(x_i)$ and $\pi$ is a projective morphism of algebraic varieties over $C\setminus \{x_2\}$.
There exists a line bundle $\LL\in\mathrm{Pic}(\XX)$ that is algebraic and $\pi$-ample over $C\setminus \{x_2\}$.
\end{prop}
\begin{proof}
We embed $X_1$ into $\bbP^N = \bbP H^0(X_1,L)$ and denote by $P$ its Hilbert polynomial.
Let $\HH$ be the Hilbert scheme $\mathrm{Hilb}^P(\bbP^N)$ with the reduced scheme structure
and let $\psi\colon\tilde{\XX}\to \HH$ be the restriction to $\HH$ of the universal family.
%We assume that $\HH$ is reduced, otherwise we replace it by its reduction.
Let $U\subset \HH$ be the open subset over which the morphism $\psi$ is smooth.
If $U$ is disconnected, we replace it by the connected component containing $[X_1]$.
Let $\mu\colon\tilde{U}\to U$ be the universal covering, and let 
$\psi'\colon\tilde{\XX}'\to \tilde{U}$ be the pull back of the family $\tilde{\XX}$ to $\tilde{U}$.
The local system $R^2\psi'_*\bbZ$ is trivial and by parallel transport its fibres
can be identified with $H^2(X_1,\bbZ)$. Using the marking $\varphi_1$
we identify the fibres of the latter local system with $\Lambda$, and the
relative Hodge to de Rham spectral sequence induces an embedding of vector
bundles $\rho\colon\psi'_*(\Omega^2_{\tilde{\XX}'/\tilde{U}})\hrarr \Lambda \otimes \OO_{\tilde{U}}$.
The family $\tilde{\XX}'$ is polarized, the class of the polarization being $h\in \Lambda$,
hence $\rho$ factors via $\Lambda_h\otimes \OO_{\tilde{U}}$. So $\psi'_*(\Omega^2_{\tilde{\XX}'/\tilde{U}})$
is a rank one subbundle of $\Lambda_h\otimes \OO_{\tilde{U}}$, and this gives us
a period map from $\tilde{U}$ to $\DD_h$. We can find a torsion-free
arithmetic subgroup $\Gamma\subset {\rm O}(\Lambda_h,q)$ and a finite covering $U'\to U$
such that the period map descends to a morphism $\mu\colon U'\to \DD_h/\Gamma$ of quasi-projective
varieties. Note that the morphism $\mu$ is dominant (see e.g. the proof of \cite[Lemma 4.5]{So1}).

Let $p_1$ and $p_2$ be the images of $(X_1,\varphi_1)$ and $(X_2,\varphi_2)$ in $\DD_h/\Gamma$.
By construction, $p_1 \in \im(\mu)$. We can find a smooth quasi-projective curve $C_1\subset \DD_h/\Gamma$
such that $p_1,p_2\in C_1$. For a very general point of $\DD_h$, the corresponding hyperk\"ahler
manifold has Picard group of rank one (generated by $h$), and we can assume that the same
is true for a very general point of $C_1$. Since $\mu$ is dominant, there exists
a curve $C_2\subset U'$ that maps dominantly to $C_1$. Taking the normalization of $C_1$ in
the function field of $C_2$, we get a curve $C_3$ and a finite morphism $\nu\colon C_3\to C_1$.
By construction, there exists a rational map from $C_3$ to $\HH$. Since $\HH$ is a projective
variety, this map extends to a morphism $\xi\colon C_3\to \HH$.
We obtain a projective family $\XX' = C_3 \times_\HH \tilde{\XX}$ over $C_3$.
$$
\begin{tikzcd}
%C_3 \arrow[hookleftarrow]{r}\arrow{d}{\xi}\arrow[bend left]{rr}{\nu} & C_2 \arrow[hook]{d}\arrow{r} & C_1 \arrow[hook]{d} \\
\, & C_3 \arrow[dashed]{d}\arrow{r}{\nu} & C_1 \arrow[hook]{d} \\
\HH \arrow[leftarrow]{ur}{\xi}\arrow[leftarrow]{r} & U' \arrow{r}{\mu} & \DD_h/\Gamma
\end{tikzcd}
$$

Let $q_1, q_2\in C_3$ be two points with $\nu(q_i) = p_i$. Note that the fibre of $\XX'$ over
$q_1$ is isomorphic to $X_1$ by construction. The fibre over $q_2$ might be non-smooth or
not isomorphic to $X_2$. We use Lemma \ref{lem_mod} to modify the family
$\XX'$ over a disk around $q_2$, and produce a new family with fibre $X_2$.
Let $\Delta\subset C_3$ and $\Delta'\subset C_1$ be two small disks around $q_2$ and $p_2$
such that $\nu(\Delta)\subset \Delta'$ and the covering map $\DD_h\to \DD_h/\Gamma$
splits over $\Delta'$. Let $\Delta''\subset \DD_h$ map isomorphically onto $\Delta'$ under
the covering map. We obtain a morphism $\gamma\colon \Delta\to \Delta''\subset \DD_h$.
Let $\Delta^* = \Delta\setminus \{q_2\}$ and note that $\gamma|_{\Delta^*}$ is the period map for $\XX'|_{\Delta^*}$
for some choice of the marking. We can now apply Lemma \ref{lem_mod}. After passing to
some finite ramified cover $\alpha\colon C_4 \to C_3$, we obtain a smooth
family $\YY\to \alpha^{-1}(\Delta)$ with central fibre $X_2$ such that its restriction
to $\alpha^{-1}(\Delta^*)$ is isomorphic to the restriction of $\XX'' = C_4\times_{C_3}\XX'$. We modify
$\XX''$ over $\alpha^{-1}(\Delta)$ by gluing in $\YY$, and obtain
a new family $\XX\to C_4$ that contains both $X_1$ and $X_2$ as fibres. We restrict to an
open subset $C\subset C_4$ over which this family is smooth. Since the family $\XX''$ is projective,
there exists a relatively ample line bundle $\LL''\in \mathrm{Pic}(\XX'')$. Using the last claim
in Lemma \ref{lem_mod}, after gluing in $\YY$ we get a line bundle $\LL\in \mathrm{Pic}(\XX)$ whose
restriction to all fibres except possibly $X_2$ is ample. This completes the proof.
\end{proof}

\begin{lem}\label{lem_motives}
Assume that $\pi\colon \XX\to C$ is a smooth family of hyperk\"ahler manifolds
such that: $C$ is a smooth quasi-projective curve; all fibres of $\pi$ are projective;
$\pi$ is a projective morphism of algebraic varieties over a dense Zariski-open subset $U\subset C$;
there exits a line bundle $\LL\in\mathrm{Pic}(\XX)$ that is algebraic and $\pi$-ample over U.
Assume that the Andr\'e motive of the fibre $\XX_{b_0}$ is abelian for some $b_0\in C$.
Then for any $b_1\in C$ the Andr\'e motive of $\XX_{b_1}$ is abelian.
\end{lem}
\begin{proof}
Using part (1) of Corollary \ref{cor_cover} and possibly replacing $C$ by a finite cover,
we may assume that $\pi_1(C,b_0)$ acts on $H^{\sdot}(\XX_{b_0},\bbQ)$ via a homomorphism
$\rho\colon \pi_1(C,b_0)\to \Spin(V,q)_\bbQ$, where $V = H^2(\XX_{b_0},\bbQ)$
and $q$ is the Beauville-Bogomolov-Fujiki form.
Since the morphism $\pi$ is projective over an open subset of $C$, there exists
a line bundle defining the polarization. The first Chern class of this line
bundle gives a monodromy-invariant element $h\in V$ such that $q(h)>0$.
Hence we can use parts (2) and (3) of Corollary \ref{cor_cover}
to obtain a projective family $\psi\colon \AA\to C$ of Kuga-Satake abelian varieties.

Consider the product $\YY = \XX\times_C\AA$ and denote by $\xi\colon\YY \to C$ the induced morphism.
The embeddings $\bar{\nu_i}$ from Corollary \ref{cor_cover} can be viewed as global
sections of the local system $R^{2n+2d}\xi_*\bbQ(n+d)$. We need to prove that the
corresponding cohomology class $\bar{\nu}_{i,b_1}\in H^{2n+2d}(\XX_{b_1} \times \AA_{b_1},\bbQ(n+d))$
is motivated. By Proposition \ref{prop_def}, it is enough to prove that
$\bar{\nu}_{i,b_0}$ is motivated. But the fibre $\YY_{b_0}\simeq \XX_{b_0} \times \AA_{b_0}$
has abelian Andr\'e motive, hence any cohomology class on it is motivated by \cite[section 6]{A1}.
\end{proof}

\subsection{Proof of theorem \ref{thm_main}}\label{sec_thm}
We use the notation introduced in section \ref{sec_moduli}.
We choose two markings $\varphi_i\colon H^2(X_i,\bbZ)\stackrel{\sim}{\to} \Lambda$, so that
$(X_i,\varphi_i)\in \MM$. We choose very ample line bundles on $X_i$ and denote by $h_i\in \Lambda$ their classes.
We use Proposition \ref{prop_curve} to connect $X_1$ and $X_2$ by several smooth
families of hyperk\"ahler manifolds and then apply Lemma \ref{lem_motives} to these families.
We connect $X_1$ to $X_2$ in several steps, depending on the relative position of
the divisors $\DD_{h_1}$ and $\DD_{h_2}$ inside $\DD$.

{\it Case 1.} Assume that $\rho(X_1,\varphi_1)\in \DD_{h_1}\cap \DD_{h_2}$ or $\rho(X_2,\varphi_2)\in \DD_{h_1}\cap \DD_{h_2}$.
In this case we can apply Proposition \ref{prop_curve} to construct a family connecting
$X_1$ and $X_2$.

{\it Case 2.} Assume that $\mathrm{sign}(q|_{\langle h_1,h_2\rangle}) = (1,1)$. This condition implies
that $\DD_{h_1}\cap \DD_{h_2} \neq \emptyset$. By the surjectivity of
the period map $\rho\colon \MM\to \DD$ (see \cite[Theorem 5.5]{Hu2}), we can pick $(X_3,\varphi_3)\in \MM$
such that $\rho(X_3,\varphi_3)\in \DD_{h_1}\cap \DD_{h_2}$ and reduce to {\it Case 1} above.

{\it Case 3.} Assume the $q|_{\langle h_1,h_2\rangle}$ is positive definite.
In this case $\DD_{h_1}\cap \DD_{h_2} = \emptyset$, but we will find $h_3\in \Lambda$ such that $q(h_3) > 0$ and
$\DD_{h_1}\cap \DD_{h_3} \neq \emptyset$, $\DD_{h_2}\cap \DD_{h_3} \neq \emptyset$, reducing
to {\it Case 2}. Consider the set
$$
\VV = \{v\in \Lambda\otimes\bbR\st q(v)>0,\, \mathrm{sign}(q|_{\langle h_1,v\rangle}) = \mathrm{sign}(q|_{\langle h_2,v\rangle}) = (1,1)\}.
$$
This set is an open cone in $\Lambda\otimes\bbR$, and it suffices to prove that $\VV\neq \emptyset$.
Choose three vectors $e_1,e_2,e_3 \in \Lambda\otimes\bbR$ such that: $q(e_1)=q(e_2)=1$, $q(e_3)=-1$;
$e_i$ are pairwise orthogonal; $h_1 = a e_1$, $h_2 = b e_1 + c e_2$ for some $a,b,c\in \bbR$.
If $b = 0$, then $v = e_1 + e_2 + d e_3 \in \VV$ for $1 < d^2 < 2$. If $b\neq 0$, then
$v = b e_1 + c e_2 + d e_3\in \VV$ for $c^2 < d^2 < b^2+c^2$.

{\it Case 4.} Assume the $q|_{\langle h_1,h_2\rangle}$ is degenerate. Then we will find
$h_3\in \Lambda$ such that $q(h_3) > 0$ and the restrictions $q|_{\langle h_1,h_3\rangle}$,
$q|_{\langle h_2,h_3\rangle}$ are non-degenerate, reducing to the previous cases.
Consider the set
$$
\VV = \{v\in \Lambda\otimes\bbR\st q(v)>0,\, q|_{\langle h_1,v\rangle} \mbox{ and } q|_{\langle h_2,v\rangle} \mbox{ non-degenerate}\}.
$$
As above, $\VV$ is an open cone and we need to check that $\VV\neq\emptyset$.
The condition that $q|_{\langle h_i,v\rangle}$ is degenerate is given by the vanishing of the
determinant of the Gram matrix, hence it defines a hypersurface in $\Lambda\otimes \bbR$.
Since the set $\{v\in \Lambda\otimes \bbR\st q(v)>0\}$ is clearly open and non-empty,
$\VV$ is also non-empty.
This completes {\it Case 4},
and the proof of the theorem.

\section*{Acknowledgements}

Thanks to Ben Moonen for communicating to me the problem of
lifting the Kuga-Satake embedding to the category of Andr\'e motives,
to Salvatore Floccari for his interest in this work
and for pointing out the relevance of \cite[Theorem 5.1]{Fl} for Corollary \ref{cor_MT},
to Giovanni Mongardi for the useful discussion of O'Grady's manifolds.

\end{document}